\documentclass{amsart}

\usepackage{amssymb, amscd, amsmath}

\newtheorem{theorem}{Theorem}[section]
\newtheorem{proposition}[theorem]{Proposition}
\newtheorem{corollary}[theorem]{Corollary}
\newtheorem{lemma}[theorem]{Lemma}

\theoremstyle{remark}
\newtheorem{remark}[theorem]{\bf Remark}

\newcommand{\onto}[1]{\stackrel{#1}{\to}}

\newcommand{\inj}{\hookrightarrow}

\newcommand{\cO}{\mathcal{O}}
\newcommand{\fa}{\mathfrak{a}}

\newcommand{\bn}{\bigskip\noindent}

\newcommand{\Pm}{(\mathrm{P}_m)}
\newcommand{\PPm}{(\mathrm{P}'_m)}
\newcommand{\Pem}{(\mathrm{P}^e_m)}
\newcommand{\Tem}{(\mathrm{T}^e_m)}
\newcommand{\ii}{\textbf{i}}
\newcommand{\uu}{\textbf{u}}
\newcommand{\ff}{\widetilde{\varphi}}
\newcommand{\fp}{\widetilde{\psi}}
\newcommand{\mm}{\mathfrak{m}}
\newcommand{\mZ}{\mathbb{Z}}
\newcommand{\res}{\mathrm{Res}}

\title[\tiny An ultrametric space of Eisenstein polynomials and ramification theory]
{An ultrametric space of Eisenstein polynomials \\and ramification theory}

\author{Manabu Yoshida}
\address{
	Graduate School of Mathematics, Kyushu University,
	Fukuoka 819-0395, Japan
}
\email{m-yoshida@math.kyushu-u.ac.jp}
\thanks{The author is supported by the Japan Society for the Promotion 
of Scientist Fellowships for Young Scientists.}
\date{}

\begin{document}

\maketitle

\begin{abstract}
We give a criterion whether given Eisenstein polynomials 
over a local field $K$ define the same extension over $K$ 
in terms of a certain non-Archimedean metric on the set of 
polynomials. The criterion and its proof depend on 
ramification theory. 
\end{abstract}
\section{Introduction}
\quad Let $K$ be a complete discrete valuation field, 
$k$ its residue field (which may be imperfect) of characteristic $p>0$, 
$v_K$ its valuation normalized by $v_K(K^{\times})=\mathbb{Z}$, 
$\cO_K$ its valuation ring, 
$\Omega$ a fixed algebraic closure of $K$ and 
$\bar{K}$ the separable closure in $\Omega$. 
The valuation $v_K$ can be extended to $\Omega$ uniquely and the 
extension is also denoted by $v_K$. 
Let $L/K$ be a finite Galois extension with ramification index $e$ and 
inertia degree $1$. 
Denote by $\cO_L$ the integral closure of $\cO_K$ in $L$. 
Take a uniformizer $\pi_L$ of $\cO_L$ and its minimal polynomial $f$ over $K$, 
which is an Eisenstein polynomial over $K$. 
Let $E_K^e$ be the set of all Eisenstein polynomials over $K$ of degree $e$. 
For two polynomials $g=\sum a_i X^i, h=\sum b_i X^i \in E_K^e$, we put 
\[ v_K(g,h):=\min_{0 \leq i \leq e-1} \{v_K(a_i - b_i) + \frac{i}{e} \}.\] 
Then the function $v_K(\cdot,\cdot)$ defines a non-Archimedean metric on 
$E_K^e$ (\emph{cf}.\ Lem.\ \ref{DistanceEisenstein}). 
For any $g \in E_K^e$, 
we put $M_g=K(\pi_g)$, where $\pi_g$ is a root of $g$. 
For any real number $m \geq 0$, 
we consider the following property: 
\begin{quote}
\begin{itemize}
\item[$\Tem$] \emph{For any $g \in E_K^e$, 
if $v_K(f,g) \geq m$, 
then there exists a $K$-isomorphism $L \cong M_g$}.
\end{itemize}
\end{quote}
This property does not depend on the choice of $\pi_L$ 
(\emph{cf}.\ Prop.\ \ref{Tem=Pem}). 
Let $u_{L/K}$ be the largest upper numbering ramification break of 
$L/K$ in the sense of \cite{Fontaine85} (\emph{cf}.\ Sect.\ \ref{Ramification}). 
Then we can prove the following (Prop.\ \ref{KrasnerLemmaRenewB}): 
\begin{proposition}
The property $\Tem$ is true for $m > u_{L/K}$, 
and is not true for $m \leq u_{L/K} - e^{-1}$. 
\end{proposition}
This proposition is a consequence of 
results of Fontaine on a certain property 
$\Pm$ (\emph{cf}.\ Appendix). 
Since both $v_K(f,g)$ and $u_{L/K}$ are in $e^{-1} \mathbb{Z}$, 
the truth of $\Tem$ is constant for $u_{L/K}-e^{-1} < m \leq u_{L/K}$.  
Therefore, we want to know the truth of $\Tem$ for $m=u_{L/K}$. 
The property $\Tem$ behaves mysteriously at the break $m=u_{L/K}$. 
It depends on the ramification of $L/K$ and the residue field $k$. 
Our main theorem in this paper is the following (Cor.\ \ref{CorMainTheoremB}): 

\bn
{\bf Theorem A}. 
\emph{If $L/K$ is tamely ramified, then $\Tem$ is true for $m=u_{L/K}$ 
if and only if the residue field $k$ has no cyclic extension of degree $e$. 
If $L/K$ is wildly ramified, 
then $\Tem$ is true for $m=u_{L/K}$ if and only if 
the residue field $k$ has no cyclic extension of degree $p$}. 

\bn
We reduce the proof of this theorem to the abelian case 
by showing that $\Tem$ is equivalent to a certain 
property $\Pem$, which has such a reduction property (Prop.\ \ref{reduce}). 
To prove the abelian case, we show that, 
by using the properties of the ultrametric space $E_K^e$, 
the truth of the property $\Tem$ for 
$m=u_{L/K}$ is equivalent to the surjectivity of the norm map 
$N_{m-1}:U_L^{\psi(m-1)}/U_L^{\psi(m-1)+1} \to U_K^{m-1}/U_K^m$ 
between the graded quotients of the higher unit groups of $L$ and $K$, 
where $\psi$ is the Hasse-Herbrand function of $L/K$ (Prop.\ \ref{ProofAbelianCase}). 
Finally, we calculate its cokernel by using the well-known exact sequence 
(Prop.\ \ref{CokerNorm})
\[0 \to G_{\psi(m-1)}/G_{\psi(m-1)+1} \to U_L^{\psi(m-1)}/U_L^{\psi(m-1)+1} 
\overset{N_{m-1}}{\longrightarrow} U_K^{m-1}/U_K^m, \]
where $G_{i}$ is the $i$th lower numbering ramification group 
in the sense of \cite{Serre79} (\emph{cf}.\ Rem.\ \ref{Numbering}). 
The vanishing of Coker($N_{m-1}$) is equivalent to the conditions 
in Theorem A.  

Our results are useful for computations to 
construct explicit extensions over $K$ which satisfy given conditions. 
For example, such computations are required 
in \cite{Suzuki}, \cite{Yoshida11} and \cite{Yoshida12}. 
Indeed, the proof of Proposition 5.1,\ (1) in \cite{Suzuki} is based on our results. 
In \cite{Yoshida11} and \cite{Yoshida12}, 
our approaches are used to identify totally ramified extensions over $\mathbb{Q}_p$.

\bn
\emph{Plan of this paper}. 
In Section \ref{Ramification}, 
we give a review of the classical ramification theory. 
In Section \ref{Distance}, 
we recall a notion of ultrametric space on polynomials. 
In Section \ref{ThePropertyTem}, 
we define the property $\Tem$, 
which is the main object in this paper. 
In Section \ref{MainTheorem}, 
we state our main theorem and its consequences. 
In Section \ref{ProofMainTheorem}, 
we prove the main theorem. 
In the Appendix, we consider similar properties $\PPm$ and $\Pm$. 
To remove confusion, 
we clarify the relation between the four properties which appear in this paper:
\[ \Pm \iff \PPm \Longrightarrow \Pem \iff \Tem, \]
where the last equivalence requires the condition $m>1$.

\bn
\emph{Notations}. 
We fix an algebraic closure $\Omega$ of $K$ and denote by $\bar{K}$ the 
separable closure of $K$ in $\Omega$.
We denote by $\cO_K$, $\mm_K$, $\pi_K$ and $v_K$,
respectively, the valuation ring of $K$, its maximal ideal,
a uniformizer of $K$ and the valuation on $K$ normalized by $v_K(K^{\times})=\mZ$. 
We assume throughout that all algebraic extensions of $K$ under discussion 
are contained in $\Omega$. 
The valuation $v_K$ of $K$ extends to $\Omega$ 
uniquely and the extension is also denote by $v_K$. 
If $M$ is an algebraic extension of $K$, 
then we denote by $\cO_M$ 
the integral closure of $\cO_K$ in $M$, and by $\mm_M$ the maximal ideal of $\cO_M$. 
For any integer $n \geq 1$, 
we put $U_K^n=1+\mm_K^n$ and $U_L^n=1+\mm_L^n$. 
Put $U_K^0=\cO_K^{\times}$ and 
$U_L^0=\cO_L^{\times}$ by convention.

\bn
\emph{Conventions}. 
Throughout this paper, 
we assume that $L/K$ is unferociously ramified\footnote{
We mean by an \emph{unferociously} ramified extension 
an algebraic extension whose residue field extension is separable. 
} extension. 
We do not consider the trivial case $L=K$. 

\bn
\emph{Acknowledgments}.\ 
The author thanks Takashi Suzuki for communicating 
Proposition \ref{CokerNorm} to him. 
He also thanks Kazuya Kato for suggesting an improvement of the proof 
of Lemma \ref{hominjective}. 
Finally, he thanks Yuichiro Taguchi for pointing out Remark \ref{separability}.

\section{Ramification theory}\label{Ramification}

\quad In this section, we recall the classical ramification theory for 
Galois extensions of $K$. 
Our notations are based on \cite{Fontaine85}, Section 1. 
Let $L$ be a finite Galois extension of $K$ with Galois group $G$. 
There exists an element $\alpha \in \cO_L$ such that $\cO_L = \cO_K[\alpha]$ 
(the existence of such an element is proved in 
\cite{Serre79}, Chap.\ III, Sect.\ 6, Prop.\ 12). 
The order function $\ii_{L/K}$ is defined on $G$ by 
\[\ii_{L/K}(\sigma)=\inf_{a \in \cO_L}
 v_K(\sigma(a)-a) = v_K(\sigma \alpha - \alpha)\]
for any $\sigma \in G$. 
Then the $i$th \emph{lower numbering ramification group $G_{(i)}$ of $G$} 
are defined for a real number $i \geq 0$ by 
\[G_{(i)}=\{ \sigma \in G\ |\ \ii_{L/K}(\sigma) \geq i \}.\]
The transition function $\ff_{L/K}:\mathbb{R}_{\geq 0} \to 
\mathbb{R}_{\geq 0}$ of $L/K$ is defined by 
\[\ff_{L/K}(i)=\int_0^i \sharp G_{(t)} dt\]
for any real number $i \geq 0$, 
where $\sharp G_{(t)}$ is the cardinality of $G_{(t)}$. 
This is a piecewise-linear, monotone increasing function, mapping the interval $[0,+\infty)$ 
onto itself. Its inverse function is denoted by $\fp_{L/K}$. 
The following lemma is a fundamental property of these functions:
\begin{lemma}[\cite{Fontaine85}, Prop.\ 1.4]\label{Herbrand}
Let $L$ be a finite Galois extension of $K$. 
Let $f$ be the minimal polynomial of $\alpha$ over $K$ and $\beta$ an element of $\Omega$. 
Put $i = \sup_{\sigma \in G} v_K(\sigma (\alpha) - \beta)$ and 
$u= v_K(f(\beta))$. 
Then we have
\[ u = \ff_{L/K}(i),\quad \fp_{L/K}(u)=i. \]
\end{lemma}
We define the order function $\uu_{L/K}$ on $G$ by 
\[\uu_{L/K}(\sigma)=\ff_{L/K}(\ii_{L/K}(\sigma)) \]
for any $\sigma \in G$. 
Then the $u$th \emph{upper numbering ramification group $G^{(u)}$ of $G$} 
are defined for a real number $u \geq 0$ by 
\[G^{(u)}=\{\sigma \in G\ |\ \uu_{L/K}(\sigma) \geq u \}.\]
\begin{remark}\label{Numbering}
Denote by $G_i$, $G^u$, $\varphi_{L/K}$ and $\psi_{L/K}$, 
respectively, the $i$th lower numbering ramification group, 
the $u$th upper numbering ramification group, 
the transition function and its inverse function in the sense of 
\cite{Serre79}, Chapter IV. 
The relation between our notations and those of \cite{Serre79} is 
the following: For any real number $i, u \geq -1$, we have
\[ G_i = G_{((i+1)/e)},\quad G^u = G^{(u+1)} \]
and
\[ \varphi_{L/K}(i) = \ff_{L/K}((i+1)/e)-1,\quad 
\psi_{L/K}(u) = e \fp_{L/K}(u+1)-1, \]
where $e$ is the ramification index of $L/K$. 
\end{remark}
Denote the largest lower (resp.\ upper) numbering ramification break by 
\[ i_{L/K}=\inf \{ i \in \mathbb{R}\ |\ G_{(i)}=1 \},\quad 
u_{L/K}=\inf \{ u \in \mathbb{R}\ |\ G^{(u)}=1 \}. \]
The graded quotients of $(G^{(u)})_{u \geq 1}$ are abelian 
and killed by $p$ (\cite{Serre79}, Chap.\ IV, Sect.\ 2, Cor.\ 3). 
In particular, $G^{(u)}$ is abelian and killed by $p$ for $u=u_{L/K}$ 
if $L/K$ is wildly ramified. 

We assume that $L/K$ is totally ramified.
If there is no confusion, we write $\psi=\psi_{L/K}$ for simplicity. 
\begin{proposition}[\cite{Serre79}, Chap.\ V, Prop.\ 8]
For any integer $n \geq 0$, $N(U_L^{\psi(n)}) \subset U_K^n$ 
and $N(U_L^{\psi(n)+1}) \subset U_K^{n+1}$,
where $N:=N_{L/K}$ is the norm map.
\end{proposition}
This proposition allows us, 
by passage to the quotient, 
to define the homomorphisms
\[ N_n:U_L^{\psi(n)}/U_L^{\psi(n)+1} \to U_K^n/U_K^{n+1} \quad (n \geq 0). \]
The homomorphism $N_n$ is a non-constant polynomial map 
(\cite{Serre79}, Chap.\ V, Sect.\ 6, Prop.\ 9). 
\begin{proposition}[\cite{Serre79}, Chap.\ V, Sect.\ 6, Prop.\ 9]\label{exact}
For any integer $n \geq 0$, the following sequence is exact:
\[0 \to G_{\psi(n)}/G_{\psi(n)+1} \onto{\theta_n} U_L^{\psi(n)}/U_L^{\psi(n)+1} 
\overset{N_n}{\longrightarrow} U_K^n/U_K^{n+1}, \]
where $\theta_n$ is defined by $\sigma \mapsto \sigma(\pi_L)/\pi_L$. 
\end{proposition}
\begin{remark}\label{separability}
The polynomial $N_n$ is separable since $\theta_n$ is injective. 
Hence if the residue field $k$ is separably closed, then we have 
\[ 0 \to G_{\psi(n)}/G_{\psi(n)+1} \onto{\theta_n} U_L^{\psi(n)}/U_L^{\psi(n)+1} 
\overset{N_n}{\longrightarrow} U_K^n/U_K^{n+1} \to 0. \]
\end{remark}

\section{An ultrametric space of monic irreducible polynomials}\label{Distance}

\quad In this section, we define a non-Archimedean metric 
on the set, denoted by $P_K$, of all monic irreducible polynomials over $K$. 
For $f,g \in P_K$, we denote by $\res(f,g)$ the resultant of $f$ and $g$. 
Then $v_K(\res(\cdot,\cdot))$ defines a non-Archimedean metric on $P_K$ 
(see \cite{Krasner66} or \cite{Pauli01}, Sect.\ 4 for the proofs). 
It is well-known that 
\[ v_K(\res(f,g)) = 
\sum_{i,j} v_K(\alpha_i - \beta_j) = \mathrm{deg}(f) v_K(f(\beta)) = 
\mathrm{deg}(g) v_K(g(\alpha)) , \]
where $\alpha_i$ (resp.\ $\beta_j$) runs through the all roots of 
$f$ (resp.\ $g$) and $\alpha$ (resp.\ $\beta$) is a root of $f$ (resp.\ $g$). 
The third and forth presentations are independent of the choice of 
roots by the irreducibility of polynomials. 
Denote by $E_K^e$ the set of all Eisenstein polynomials over $K$ of degree $e$. 
For $f$, $g \in E_K^e$, we put 
\[ v_K(f,g)=e^{-1} v_K(\res(f,g))=v_K(f(\pi_g)), \]
where the last equality follows from the above equality. 
Then the function $v_K(\cdot,\cdot)$ also defines a non-Archimedean metric on $E_K^e$. 
There is a useful formula to calculate the metric on $E_K^e$: 
\begin{lemma}[\cite{Krasner66}]\label{DistanceEisenstein}
Let $f,g \in E_K^e$.
Write $f(X)$ $=$ $X^e$ $+$ $a_{e-1}X^{e-1}$ $+$ $\cdots$ $+$ $a_0$ and 
$g(X)$ $=$ $X^e$ $+$ $b_{e-1}X^{e-1}$ $+$ $\cdots$ $+$ $b_0$.
Then we have
\[ v_K(f,g) = \min_{0 \leq i \leq e-1} \{ v_K(a_i - b_i) + \frac{i}{e} \}. \]
\end{lemma}

\section{The property $\Tem$}\label{ThePropertyTem}

In this section, we define the property $\Tem$ and 
determine the truth for any real number 
$m \geq 0$ except a neighborhood of the break $u_{L/K}$. 
The proofs in this section essentially depend on \cite{Fontaine85}, 
Proposition 1.5. 
Let $L/K$ be a finite Galois totally ramified extension of degree $e$. 
Take a uniformizer $\pi_L$ of $L$. 
Let $f \in E_K^e$ be the minimal polynomial of $\pi_L$ over $K$. 
For any $g \in E_K^e$, we put $M_g=K(\pi_g)$, 
where $\pi_g$ is a root of $g$. 
For any real number $m \geq 0$, 
we consider the following property:
\begin{quote}
\begin{itemize}
\item[$\Tem$] For any $g \in E_K^e$, 
if $v_K(f,g) \geq m$, 
then there is a $K$-isomorphism $L \cong M_g$.
\end{itemize}
\end{quote}
Let $u_{L/K}$ be the upper numbering ramification break of $L/K$. 
Then we have the following:
\begin{proposition}\label{KrasnerLemmaRenewB}
$\mathrm{(i)}$ If $m > u_{L/K}$, then $\Tem$ is true. 

\noindent
$\mathrm{(ii)}$ If $m \leq u_{L/K} - e^{-1}$, 
then $\Tem$ is not true. 
\end{proposition}
\begin{proof}
(i) By assumption, 
we have $v_K(f,g)=v_K(f(\pi_g)) \geq m > u_{L/K}$. 
According to Lemma \ref{Herbrand}, we have 
\[ \sup_{\sigma \in G} v_K(\pi_g - \sigma (\pi_L)) = \fp_{L/K}(v_K(f(\pi_g))) 
> \fp_{L/K}(u_{L/K}) = i_{L/K}. \] 
Hence there exists $\sigma_0 \in G$ such that 
\[ v_K(\pi_g - \sigma_0 (\pi_L)) > i_{L/K} = 
\sup_{\sigma \not= 1} v_K(\sigma (\pi_L) -\pi_L) 
=\sup_{\sigma \not= 1} v_K(\sigma (\sigma_0 (\pi_L)) - \sigma_0 (\pi_L)). \]
By Krasner's lemma, we have $L \cong K(\sigma_0 (\pi_L)) \subset K(\pi_g)=M_g$. 

\noindent
(ii) This follows from Lemma \ref{Example} below immediately. 
\end{proof}
\begin{lemma}\label{Example}
Let $g \in E_K^e$. 
If $v_K(f,g) = u_{L/K} - e^{-1}$, 
then we have $L \not\cong M_g$. 
\end{lemma}
\begin{proof}
By assumption, we have $v_K(f(\pi_L)) = v_K(f,g) = u_{L/K} - e^{-1}$. 
By Lemma \ref{Herbrand}, we have 
\[ \sup_{\sigma \in G} v_K(\pi_g - \sigma (\pi_L)) = 
\fp_{L/K}(u_{L/K} - e^{-1}) = i_{L/K} - \frac{1}{ed}, \]
where $d:=\sharp G^{(u_{L/K})}$. 
By multiplying $e$ with the above equation, we have
\[ v_L(\pi_g - \sigma_0 \pi_L) = e i_{L/K} - \frac{1}{d}, \]
for some $\sigma_0 \in G$. 
If we suppose $L = M_g$, 
then LHS is an integer. 
However, RHS is never an integer. 
This is a contradiction. 
Hence we have $L \not= M_g$. 
\end{proof}

\section{Main Theorem}\label{MainTheorem}

In this section, we state our main theorem and its consequences. 
Let $L/K$ be a finite totally ramified Galois extension of degree $e$. 

\begin{theorem}\label{MainTheoremB}
The property $\Tem$ for $m=u_{L/K}$ is equivalent to the condition 
$\mathrm{Hom}_{\mathrm{cont}}(G_k,G_{(i_{L/K})})=1$.
\end{theorem}
\begin{corollary}\label{quasi-finite}
Assume $k$ is a quasi-finite field. 
Then $\Tem$ is not true for $m=u_{L/K}$. 
\end{corollary}
\begin{proof}
By assumption, we have
\[ \mathrm{Hom}_{\mathrm{cont}}(G_k,G_{(i_{L/K})})=\mathrm{Hom}_{\mathrm{cont}}(\hat{\mathbb{Z}},G_{(i_{L/K})}) 
\cong G_{(i_{L/K})} \not= 1. \]
We obtain the desired result by Theorem \ref{MainTheoremB}. 
\end{proof}
Theorem A is a consequence of the following: 
\begin{corollary}\label{CorMainTheoremB}
If $L/K$ is tamely ramified, then $\Tem$ for $m=u_{L/K}$ is equivalent to the 
condition $k^{\times}/(k^{\times})^e = 1$. 
If $L/K$ wildly ramified, then $\Tem$ for $m=u_{L/K}$ is equivalent to the 
condition $k/\wp(k)=0$, where $\wp(X):=X^p - X$. 
\end{corollary}
\begin{proof}
Assume $L/K$ is tamely ramified. 
Then the group $G=G_{(i_{L/K})}$ 
is isomorphic to a finite cyclic group $\mu_e$ of order $e$. 
Note that $k$ contains the $e$th roots of unity. 
Hence we have 
$\mathrm{Hom}_{\mathrm{cont}}(G_k,\mu_e)$ $\cong$ $k^{\times}/(k^{\times})^e$ 
by Kummer theory. 
It follows the desired result by Theorem \ref{MainTheoremB}. 
Assume $L/K$ is wildly ramified. 
Then we have $G_{(i_{L/K})} \cong \oplus\ \mathbb{Z}/p \mathbb{Z}$. 
Therefore, we obtain 
$\mathrm{Hom}_{\mathrm{cont}}(G_k,G_{(i_{L/K})})$ $\cong$ $\oplus\ k/\wp(k)$ by 
Artin-Schreier theory. 
By Theorem \ref{MainTheoremB}, the proof completes. 
\end{proof}

\section{Proof of the main theorem}\label{ProofMainTheorem}

\subsection{Reduction to the abelian case}\label{Reduction}

In this subsection, we reduce the proof of Theorem \ref{MainTheoremB} 
to the case where $L/K$ has only one ramification break so that 
$L/K$ is abelian. 
To complete this, 
we consider the property $\Pem$ below. 
Let $L/K$ be a finite Galois totally ramified extension of degree $e$. 
\begin{quote}
\begin{itemize}
\item[$\Pem$] \emph{For any finite totally ramified extension $M/K$ of degree $e$, 
if there exists an $\cO_K$-algebra homomorphism 
$\cO_L \to \cO_M/\fa_{M/K}^m$, 
then there exists a $K$-isomorphism $L \cong M$}.
\end{itemize}
\end{quote}
\begin{proposition}\label{Tem=Pem}
The property $\Tem$ is equivalent to $\Pem$ for any real number $m>1$. 
\end{proposition}
\begin{proof}
Let $L/K$ be a finite Galois totally ramified extension of degree $e$. 
Take a uniformizer $\pi_L$ of $L$ and $f \in E_K^e$ its minimal polynomial over $K$. 
Assume that $\Tem$ is true for $f$ and $m>1$. 
Then we show that $\Pem$ is also true for $L/K$ and $m$. 
Suppose there exists an $\cO_K$-algebra homomorphism 
$\eta:\cO_L \to \cO_M/\fa_{M/K}^m$ for a totally 
ramified extension $M$ of $K$ of degree $e$. 
By Lemma \ref{hominjective} below, 
we have $v_K(\beta)=1/e$, where $\beta$ is a lift of $\eta(\pi_L)$ in $\cO_M$. 
Hence $\beta$ is a uniformizer of $M$. 
Take the minimal polynomial $g \in E_K^e$ of $\beta$ over $K$. 
By the well-definedness of $\eta$, 
we have $v_K(f,g) = v_K(f(\beta)) \geq m$. 
Since the property $\Tem$ is true for $f$ and $m$, 
we have $L = M_g \cong M$. 

Conversely, we assume that $\Pem$ is true for $L/K$ and $m>1$. 
Then we show that $\Tem$ is also true for $f$ and $m$. 
Suppose $v_K(f,g) \geq m$ for an element $g \in E_K^e$. 
Note that $v_K(f(\pi_g)) = v_K(f,g) \geq m$, where 
$\pi_g$ is a root of $g$. 
Then the map $\cO_L \to \cO_{M_g}/\fa_{M_g/K}^m$ defined by $\pi_L \mapsto \pi_g$ 
is an $\cO_K$-algebra homomorphism. 
Since $\Pem$ is true for $L/K$ and $m$, 
we have $L \cong M_g$. 
\end{proof}
\begin{lemma}\label{hominjective}
Let $L/K$ be a finite totally ramified extension of degree $e$, 
$\pi_L$ a uniformizer of $L$ 
and $m>1$ a real number. 
Assume there exists an $\cO_K$-algebra homomorphism $\eta:\cO_L \to \cO_M/\fa_{M/K}^m$ 
for an algebraic extension $M$ of $K$. 
Then we have $v_K(\beta)=1/e$, where $\beta$ is any lift of $\eta(\pi_L)$. 
\end{lemma}
\begin{proof}
Assume there exists an $\cO_K$-algebra homomorphism $\eta:\cO_L \to \cO_M/\fa_{M/K}^m$. 
Put $u=\pi_L^e/\pi_K$. 
Then $u$ is a unit, so that $\eta(u)$ is also. 
Since $\beta^e \equiv \eta(u) \pi_K$ in $\cO_M/\fa_{M/K}^m$ and $m > 1$, 
we have $v_K(\beta^e)=v_K(\pi_K)=1$. 
Hence $v_K(\beta)=1/e$. 
\end{proof}
\begin{proposition}\label{reduce}
Let $L$ be a finite Galois totally 
ramified extension of $K$ of degree $e$ and 
$K'$ the fixed field of $L$ by $H:=G^{(u_{L/K})}$. 
Take a uniformizer $\pi_L$ of $L$. 
Let $f$ $\mathrm{(}$resp.\ $f'\mathrm{)}$ 
be the minimal polynomial of $\pi_L$ over $K$ $\mathrm{(}$resp.\ $K'\mathrm{)}$. 
Then the property $\Tem$ is true for $f$ and $m=u_{L/K}$ 
if and only if $(\mathrm{T}_m^{e'})$ is true for $f'$ and $m=u_{L/K'}$, 
where $e'$ is the ramification index of $L/K'$. 
\end{proposition}
\begin{proof}
If $L/K$ is tamely ramified, 
then we have $K=K'$. 
Hence there is nothing to prove. 
Thus we may assume $L/K$ is wildly ramified, so that $m=u_{L/K} > 1$. 
By definition, $L/K'$ is also wildly ramified, 
so that $m=u_{L/K'}>1$. 
By Proposition \ref{Tem=Pem}, 
$\Tem$ for $f$ and $m=u_{L/K}$ is equivalent to $\Pem$ for $L/K$ and $m=u_{L/K}$. 
Similarly, $(\mathrm{T}_m^{e'})$ for $f'$ and $m=u_{L/K'}$ is equivalent 
to $(\mathrm{P}_m^{e'})$ for $L/K'$ and $m=u_{L/K'}$. 
Thus it is enough to 
prove that $\Pem$ for $L/K$ and $m=u_{L/K}$ is equivalent 
to $(\mathrm{P}_m^{e'})$ for $L/K'$ and $m=u_{L/K'}$. 
This is a direct consequence of the following lemma: 
\begin{lemma}[\cite{Suzuki}, Lem.\ 2.2]\label{homequivalence}
Let $L$ and $K'$ be as in Proposition \ref{reduce}. 
Let $M$ be an algebraic extension of $K$. 
The following conditions are equivalent: 

\noindent
$\mathrm{(i)}$ There exists an $\cO_K$-algebra homomorphism $\cO_L \to \cO_M/\fa_{M/K}^{u_{L/K}}$. 

\noindent
$\mathrm{(ii)}$ The field $K'$ is contained in $M$, and 
there exists an $\cO_{K'}$-algebra homomorphism $\cO_L \to \cO_M/\fa_{M/K'}^{u_{L/K'}}$. 
\end{lemma}
\end{proof}

\subsection{The proof of the abelian case}\label{ProofAbelCase}

In this subsection, 
we complete the proof of Theorem \ref{MainTheoremB}. 
It suffices to show the abelian case by Proposition \ref{reduce}. 
Then the break $u_{L/K}$ is an integer by the Hasse-Arf theorem 
(\cite{Serre79}, Chap.\ V, Sect.\ 7, Thm.\ 1). 
Therefore, it suffices to prove the integer break case: 
\begin{proposition}\label{ProofAbelianCase}
Assume $u_{L/K}$ is an integer. 
Then Theorem \ref{MainTheoremB} is true.  
\end{proposition}
\begin{proof}
Put $m=u_{L/K}$. 
Let $L/K$ be a finite Galois totally 
ramified extension of degree $e$ 
such that $u_{L/K}$ is an integer. 
Take a uniformizer $\pi_L$ of $L$ and its minimal polynomial $f$ over $K$. 
Let $g \in E_K^e$. Put $M_g=K(\pi_g)$, where $\pi_g$ is a root of $g$. 
We write $f=X^e + a_{e-1} + \cdots + a_0$ and 
$g=X^e + b_{e-1} + \cdots + b_0$. 
We want to show that if $v_K(f,g) = m$, then the equality $L=M_g$ 
is equivalent to the condition $\mathrm{Hom}_{\mathrm{cont}}(G_k,G_{(i_{L/K})})=1$. 

First, we prove that it suffices to consider $g$ of the following form 
by replacing $f$ with suitable one: 
\[ g= g_u := X^e + a_{e-1}X^{e-1} + \cdots + a_1 X + u a_0, \]
where $u$ is an element of $U_K^{m-1} \setminus U_K^m$. 
By Lemma \ref{DistanceEisenstein} and the assumption that $u_{L/K}$ is an integer, 
we have 
\[ v_K(f(\pi_g)) = \min_{0 \leq i \leq e-1} 
\{ v_K(a_i - b_i) + \frac{i}{e} \} = v_K(a_0 - b_0) = m. \]
Thus we have $b_0 = u a_0$ for some $u \in U_K^{m-1} \setminus U_K^m$.
Let $f_0 := X^e + b_{e-1} X^{e-1} + \cdots + b_1 X + a_0$. 
Note that $v_K(f,f_0) = \min_{1 \leq i \leq e-1} \{v_K(a_i - b_i) + i/e \} > m$.
According to Proposition \ref{KrasnerLemmaRenewB}, 
the extension defined by $f_0$ coincides with $L$.
By replacing $f$ with $f_0$, 
we reduce the problem to the desired situation.

Second, we show that $L=M_u$ for any $u \in U_K^{m-1} \setminus U_K^m$ if and only if 
the map $N_{m-1}$ is surjective, 
where $M_u/K$ is the extension defined by $g_u$ and 
$N_{m-1}$ is the norm map defined in Section \ref{Ramification}. 
Assume $L=M_u$ for any $u \in U_K^{m-1} \setminus U_K^m$. 
By Lemma \ref{Herbrand}, 
we have $v_K(\sigma_0 (\pi_L) - \pi_{M_u}) = i_{L/K}$ for some $\sigma_0 \in G$.
Take $u':=\pi_{M_u}/\sigma_0 (\pi_{L}) \in L=M_u$. 
Note that $v_L(1-u') = ei_{L/K}-1$ 
by the equality $v_K(\sigma_0 (\pi_L) - u' \sigma_0 (\pi_L)) = i_{L/K}$. 
and $\psi(m-1) = ei_{L/K}-1$ by Remark \ref{Numbering}. 
Thus we have $u' \in U_L^{\psi(m-1)} \setminus U_L^{\psi(m-1)+1}$. 
Moreover, we have $u a_0 = N(\pi_{M_u}) = 
N(u' \sigma (\pi_L)) = N(u') N(\sigma (\pi_L)) = N(u') a_0$, 
so that we have $N(u') = u$. 
Conversely, we assume that the map $N_{m-1}$ is surjective. 
Then there exists an element $u'$ of $U_L^{\psi(m-1)} \setminus U_L^{\psi(m-1)+1}$ such 
that $N(u') = u$. 
Note that $\pi_L' := u' \pi_L$ is a uniformizer of $L$ and 
$v_K(\pi_L' - \pi_L) = i_{L/K}$. 
Let $f'$ be the minimal polynomial of $\pi_L'$ over $K$. 
By Lemma \ref{CalculationUnit} below, 
we have $v_K(\pi_L' - \pi_L) = \sup_{\sigma \in G} v_K(\pi_L' - \sigma (\pi_L))$. 
According to Lemma \ref{Herbrand}, 
we have 
\[
\begin{split}
v_K(f,f') &= v_K(f(\pi_L')) \\
&=\ff_{L/K}(\sup_{\sigma \in G} v_K(\pi'_L - \sigma(\pi_L))) 
= \ff_{L/K}(v_K(\pi_L' - \pi_L ))\\
&= \ff_{L/K}(i_{L/K}) = u_{L/K}=m. 
\end{split}
\]
By the ultrametric inequality, 
we have 
\[ v_K(g_u,f') \geq \min \{ v_K(f,g_u),v_K(f,f') \} = m. \]
The constant term of $f'$ is the same as the one of $g_u$. 
Then we have $v_K(g_u,f') \not= m$ by Lemma \ref{DistanceEisenstein} 
and that $m$ is an integer. 
Thus we have $v_K(g_u,f') > m$. 
According to Lemma \ref{Herbrand}, 
we have $L = M_u$. 
Therefore, 
we reduce the truth of $\Tem$ for $m=u_{L/K}$ to the surjectivity 
of the map $N_{m-1}$. 
Thus it is enough to prove 
$\mathrm{Coker}(N_{m-1}) \cong \mathrm{Hom}_{\mathrm{cont}}(G_k,G_{(i_{L/K})})$. 
It follows from Lemma \ref{CokerNorm} below immediately as $n=m-1$. 
\end{proof}
\begin{remark}
In the proof of Theorem \ref{ProofAbelianCase}, 
we do not require the assumption that $L/K$ is abelian. 
We need only the assumption that $u_{L/K}$ is an integer. 
\end{remark}
\begin{lemma}\label{CalculationUnit}
We have 
\[ v_K(\pi_L' - \sigma (\pi_L)) 
\begin{cases} < i_{L/K} & \text{$\sigma \in G \setminus G_{(i_{L/K})}$} \\
= i_{L/K} & \text{$\sigma \in G_{(i_{L/K})}$}.
\end{cases}
\]
Therefore, we have $v_K(\pi_L' - \pi_L) = i_{L/K} 
= \sup_{\sigma \in G} v_K(\pi_L' - \sigma (\pi_L))$. 
\end{lemma}
\begin{proof}
Suppose $\sigma \not\in G_{(i_{L/K})}$. 
Then we have $v_K(\pi_L - \sigma (\pi_L)) < i_{L/K}$. 
Hence we have
\[ 
\begin{split}
v_K(\pi_L' - \sigma (\pi_L)) & = v_K(\pi_L' - \pi_L + \pi_L - \sigma (\pi_L)) \\
& = \min \{ v_K(\pi_L' - \pi_L),v_K(\pi_L - \sigma (\pi_L)) \} \\
& = \min \{i_{L/K},v_K(\pi_L - \sigma (\pi_L)) \} < i_{L/K}. 
\end{split}
\]
Next, we consider the case 
$\sigma \in G_{(i_{L/K})}$. 
Then we have $v_K(\pi_L - \sigma (\pi_L)) = i_{L/K}$ so that 
$\pi_L/\sigma (\pi_L) \in U_L^{ei_{L/K}-1}$. 
Note that 
\[ 
v_K(\pi_L' - \sigma (\pi_L)) = v_K(u' \frac{\pi_L}{\sigma (\pi_L)} - 1) + \frac{1}{e}. 
\]
To prove that RHS is equal to $i_{L/K}$, 
we show 
$u' \cdot \pi_L/\sigma (\pi_L) \in U_L^{ei_{L/K}-1} \setminus U_L^{ei_{L/K}}$. 
This is equivalent to 
$u' \cdot \pi_L/\sigma (\pi_L) \not\equiv 1$ in 
$U_L^{ei_{L/K}-1}/U_L^{ei_{L/K}}$. 
To prove this, 
it suffices to show $N(u' \cdot \pi_L/\sigma (\pi_L)) \not\equiv 1$ 
in $U_K^{m-1}/U_K^m$ by considering the norm map 
$N_{m-1}:U_L^{\psi(m-1)}/U_L^{\psi(m-1)+1} \to U_K^{m-1}/U_K^m$ with 
$\psi(m-1)=ei_{L/K}-1$. 
Since $N(u') \not\equiv 1 \pmod{U_K^m}$ and $N(\pi_L/\sigma (\pi_L))=1$, 
we have $N(u' \cdot \pi_L/\sigma (\pi_L)) = N(u') N(\pi_L/\sigma (\pi_L)) 
= N(u') \not\equiv 1 \pmod{U_K^m}$.  
\end{proof}

\begin{lemma}\label{CokerNorm}
Let $L$ be a totally ramified Galois extension of $K$ and $n$ 
be an integer $\geq 0$. 
Then we have 
\[ \mathrm{Coker}(N_n) \cong \mathrm{Hom}_{\mathrm{cont}}(G_k,G_{\psi(n)}/G_{\psi(n)+1}). 
\]
\end{lemma}
\begin{proof}
Let $K_0$ (resp.\ $L_0$) be the completion of the maximal unramified extension of $K$ 
(resp.\ $L$). 
Apply Proposition \ref{exact} to $L_0/K_0$. 
Then the sequence 
\[ 0 \to G_{\psi(n)}/G_{\psi(n)+1} \to 
U_{L_0}^{\psi(n)}/U_{L_0}^{\psi(n)+1} 
\to U_{K_0}^{n}/U_{K_0}^{n+1} \to 0 \]
is exact. 
The Galois group $G_k$ acts on $L_0$ and $K_0$ continuously. 
Define the action of $G_k$ on $G_{n}/G_{n+1}$ by the trivial action. 
Since $L/K$ is totally ramified, the action of $G$ on $L_0$ is 
compatible with $G_k$-action. 
Thus the above sequence is exact as continuous $G_k$-modules. 
Writing out the corresponding exact cohomology sequence, 
and taking into account that $H^1_{\mathrm{cont}}(G_k,\bar{k})=0$, 
we obtain 
\[ k \to k \to H^1_{\mathrm{cont}}(G_k,G_{\psi(n)}/G_{\psi(n)+1}) \to 0. \]
Consequently, we have $\mathrm{Coker}(N_{n})$ $\cong$ 
$H^1_{\mathrm{cont}}(G_k,G_{\psi(n)}/G_{\psi(n)+1})$. 
Since $G_k$ acts on $G_{\psi(n)}/G_{\psi(n)+1}$ trivially, 
this is equal to $\mathrm{Hom}_{\mathrm{cont}}(G_k,G_{\psi(n)}/G_{\psi(n)+1})$. 
Hence the result follows. 
\end{proof}
\begin{remark}
This lemma is a generalization of \cite{Serre79}, Chapter XV, 
Section 2, Proposition 3. 
\end{remark}

\section{Appendix}\label{ThePropertiesPmTem}

\quad Throughout this appendix, 
we assume that $k$ is perfect. 
We consider a property $\PPm$, which is similar to $\Tem$. 
We completely determine the truth of $\PPm$ by showing 
that $\PPm$ is equivalent to $\Pm$. 
Let $L$ be a finite Galois extension of $K$. 
Take an element $\alpha$ of $\cO_L$ such that $\cO_L=\cO_K[\alpha]$. 
Let $f$ be the minimal polynomial of $\alpha$ over $K$ and 
$P_K$ the set of all monic irreducible polynomials over $K$.
For any $g \in P_K$, 
we put $M_g=K(\beta)$, where $\beta$ is a root of $g$. 
Consider the following property for any real number $m \geq 0$: 
\begin{quote}
\begin{itemize}
\item[$\PPm$] For any $g \in P_K$, 
if $v_K(\res(f,g)) \geq \mathrm{deg}(g) m$, 
then there exists a $K$-embedding $L \inj M_g$.
\end{itemize}
\end{quote}
If $f$ and $g$ are contained in $E_K^e$, 
then $v_K(\res(f,g)) \geq \mathrm{deg}(g) m$ is equivalent 
to $v_K(f,g) \geq m$. 
Hence the property $\PPm$ is stronger than $\Tem$. 

For a finite Galois extension $L/K$ and real numbers $m \geq 0$, 
we consider the following property: 
\begin{quote}
\begin{itemize}
\item[$\Pm$] 
\emph{For any algebraic extension} $M/K$,\ \emph{if there exists an} $\cO_K$-\emph{algebra homomorphism} $\cO_L \to \cO_M/\fa_{M/K}^m$,\
\emph{then there exists a} $K$-\emph{embedding} $L \inj M$. 
\end{itemize}
\end{quote}
Fontaine proved the following: 
\begin{proposition}[\cite{Fontaine85}, Prop.\ 1.5]\label{Fontaine}
Let $L$ be a finite Galois extension of $K$ and $e$ the ramification index of $L/K$. 
Then there are following relations:

\noindent
$\mathrm{(i)}$ If $m>u_{L/K}$, then $(\mathrm{P}_m)$ is true.

\noindent
$\mathrm{(ii)}$ If $m\ leq u_{L/K}-e^{-1}$, then $(\mathrm{P}_m)$ is not true.
\end{proposition}
The author proved the following: 
\begin{proposition}[\cite{Yoshida09}, Prop.\ 3.4]\label{Yoshida}
Let $L$ be a finite Galois extension of $K$. 
If $m < u_{L/K}$, then $\Pm$ is not true. 
\end{proposition}
As a similar result of our main theorem, 
the truth of $\Pm$ at the ramification break depends on the ramification 
of $L/K$ and the residue field $k$: 
\begin{proposition}[\cite{Suzuki}, Thm.\ 1.1]\label{PmBreak}
Let $L$ be a finite Galois wildly ramified extension of $K$. 
Then the property $\Pm$ is true for $m=u_{L/K}$ 
if and only if $k$ has no Galois extension whose degree is divisible by $p$.
\end{proposition}
\begin{remark}
If $L/K$ is at most tamely ramified, 
then $\Pm$ is not true for $m=u_{L/K}$. 
This is shown in the proof of Proposition 3.3, \cite{Yoshida09}. 
\end{remark}

In fact, we have the following: 
\begin{proposition}\label{Equivalence}
For any real number $m \geq 0$, 
$\PPm$ is equivalent to $\Pm$. 
\end{proposition}
\begin{proof}
Let $L$ be a finite Galois extension of $K$. 
Choose an element $\alpha \in \cO_L$ such that $\cO_L=\cO_K[\alpha]$. 
Let $f$ be the minimal polynomial of $\alpha$ over $K$. 
Assume that $\PPm$ is true for $f$ and $m$. 
Then we show that $\Pm$ is also true for $L/K$ and $m$. 
Suppose 
there exists an $\cO_K$-algebra homomorphism $\eta:\cO_L \to \cO_M/\fa_{M/K}^m$ 
for an algebraic extension $M$ of $K$. 
Let $\beta$ be a lift of $\eta(\alpha)$ in $\cO_M$. 
Then we have $v_K(f(\beta)) \geq m$. 
Let $g$ be the minimal polynomial of $\beta$ over $K$. 
Then we have $v_K(\res(f,g)) = \mathrm{deg}(g) v_K(f(\beta)) 
\geq \mathrm{deg}(g)m$. 
By the property $\PPm$, there exists a $K$-embedding 
$L \inj K(\beta) \subset M$. 
Conversely, assume that $\Pm$ is true for $L/K$ and $m$. 
Then we show that $\PPm$ is also true for $L/K$ and $m$. 
Suppose $v_K(\res(f,g)) \geq \mathrm{deg}(g)m$ for an arbitrary polynomial $g \in P_K$. 
Then we have $v_K(f(\beta)) \geq m$, 
where $\beta$ is a root of $g$. 
Put $M=K(\beta)$. 
The map $\cO_L \to \cO_M/\fa_{M/K}^m$ defined 
by $\alpha \mapsto \beta$ is an $\cO_K$-algebra homomorphism. 
By the property $\Pm$, there exists a $K$-embedding $L \inj M$. 
\end{proof}
Therefore, we have the following two consequences from Propositions 
\ref{Fontaine}, \ref{Yoshida} and \ref{PmBreak}:
\begin{corollary}\label{KnownResults}
Let $L$ be a finite Galois extension of $K$. 
Then the property $\PPm$ is true for $m > u_{L/K}$, 
and is not true for $m < u_{L/K}$. 
\end{corollary}
\begin{corollary}\label{PPmBreak}
Let $L$ be a finite Galois wildly ramified extension of $K$. 
Then the property $\PPm$ is true for $m=u_{L/K}$ 
if and only if $k$ has no Galois extension whose degree is divisible by $p$.
\end{corollary}


\end{document}